\documentclass[12pt]{amsart}
\usepackage{hyperref}

\newtheorem{theorem}{Theorem}[section]
\newtheorem{definition}[theorem]{Definition}
\newtheorem{proposition}[theorem]{Proposition}
\newtheorem{lemma}[theorem]{Lemma}
\newtheorem{corollary}[theorem]{Corollary}
\newtheorem{conjecture}[theorem]{Conjecture}
\newtheorem{problem}[theorem]{Problem}
\newtheorem{question}[theorem]{Question}

\begin{document}

\title{The defect of generalized Fourier matrices}

\author{Teodor Banica}
\address{T.B.: Department of Mathematics, Cergy-Pontoise University, 95000 Cergy-Pontoise, France. {\tt teodor.banica@u-cergy.fr}}

\subjclass[2000]{05B20 (15B10, 46L37)}
\keywords{Complex Hadamard matrix, Unitary group}

\begin{abstract}
The $N\times N$ complex Hadamard matrices form a real algebraic manifold $C_N$. We have $C_N=M_N(\mathbb T)\cap\sqrt{N}U_N$, and following Tadej and \.Zyczkowski we investigate here the computation of the enveloping tangent space $\widetilde{T}_HC_N=T_HM_N(\mathbb T)\cap T_H\sqrt{N}U_N$, and notably of its dimension $d(H)=\dim(\widetilde{T}_HC_N)$, called undephased defect of $H$. Our main result is an explicit formula for the defect of the Fourier matrix $F_G$ associated to an arbitrary finite abelian group $G=\mathbb Z_{N_1}\times\ldots\times\mathbb Z_{N_r}$. We also comment on the general question ``does the associated quantum permutation group see the defect'', with a probabilistic speculation involving Diaconis-Shahshahani type variables.
\end{abstract}

\maketitle

\tableofcontents

\section*{Introduction}

A complex Hadamard matrix is a square matrix $M\in M_N(\mathbb C)$, all whose entries are on the unit circle, $|H_{ij}|=1$, and whose rows are pairwise orthogonal. If we denote by $\mathbb T$ the unit circle, and by $U_N$ the unitary group, the set formed by such matrices is:
$$C_N=M_N(\mathbb T)\cap\sqrt{N}U_N$$

Observe that $C_N$ is a real algebraic manifold. In general $C_N$ is not smooth, nor it is a complex algebraic manifold. Its structure is in fact very complicated. In order to describe it, best is to introduce the quotient set $C_N\to E_N$ consisting of complex Hadamard matrices under the standard equivalence relation between such matrices, obtained by permuting rows and columns, or multiplying them by complex numbers of modulus 1. At $N=2,3,4,5$ this set $E_N$ is well understood, thanks to the results of Haagerup in \cite{ha1}. At $N=6$, however, this set $E_N$ is only known to contain:
\begin{enumerate}
\item Two tori $\mathbb T^2$, intersecting at the Fourier matrix $F_6$.

\item A circle $\mathbb T$, coming from the Haagerup matrix $H_6^q$.

\item An isolated matrix $T_6$, discovered by Tao in \cite{tao}. 

\item A circulant matrix $C_6$, found by Bj\"orck and Fr\"oberg in \cite{bfr}.

\item An algebraic curve $S_6^\theta$, found by Beauchamp and Nicoara in \cite{bni}.

\item Some other components, all quite poorly understood, see \cite{krs}, \cite{szo}, \cite{tz1}.
\end{enumerate}

It is quite unclear, starting from this data, on how to develop a systematic study to $C_N$. An interesting idea, however, comes from the paper of Tadej and \.Zyczkowski \cite{tz2}. Motivated by some early work in \cite{kar}, \cite{nic} and by questions regarding unistochastic matrices \cite{be+}, they computed the ``enveloping tangent space'' at $H\in C_N$, given by:
$$\widetilde{T}_HC_N=T_HM_N(\mathbb T)\cap T_H\sqrt{N}U_N$$

The description of $\widetilde{T}_HC_N$ found in \cite{tz2} is of course quite a theoretical one, in terms of solutions of a certain system of linear equations associated to $H$. Of particular interest is the dimension of the space of solutions, called ``undephased defect'' of $H$:
$$d(H)=\dim(\widetilde{T}_HC_N)$$

Tadej and \.Zyczkowski did as well in \cite{tz2} a number of defect computations. Their main result there is that, for the Fourier matrix of order $N=p_1^{a_1}\ldots p_s^{a_s}$, we have:
$$d(F_N)=N\prod_{i=1}^s\left(1+a_i-\frac{a_i}{p_i}\right)$$

Back to the general case now, knowing $\widetilde{T}_HC_N$ is of course quite far from understanding the local structure of $C_N$ around a given point $H\in C_N$. As pointed out by Barros e S\'a and Bengtsson in \cite{bbe}, even in the case of the Fourier matrix $F_{12}$ this data is not enough, and finding the local structure of $C_{12}$ around $F_{12}$ would definitely require more work.

In this paper we build on the work of Tadej and \.Zyczkowski \cite{tz2}, by systematically studying the defect of complex Hadamard matrices. We will first recast the main results in \cite{tz2} into our present algebraic geometric language. Then we will investigate the defect of the Fourier matrix $F_G$ of an arbitrary finite abelian group $G=\mathbb Z_{N_1}\times\ldots\times\mathbb Z_{N_r}$. By using the same method as in \cite{tz2}, we will first obtain the following formula:
$$d(F_G)=\sum_{g\in G}\frac{|G|}{ord(g)}$$

In order to compute this quantity, we use two observations, namely the fact that this quantity is multiplicative over the isotypic components of $G$, and the fact that for a product of $p$-groups, the number $c_k$ of elements of order $\leq p^k$ is multiplicative as well.

These observations will lead to the following general formula, valid for any finite abelian group $G$, and which is our main result in this paper:
$$d(F_G)=|G|\prod_p\left( 1+\sum_{k=1}^rp^{(r-k)a_{k-1}+(a_1+\ldots+a_{k-1})-1}(p^{r-k+1}-1)[a_k-a_{k-1}]_{p^{r-k}}\right)$$

Here the numbers $a_0=0$ and $a_1\leq a_2\leq\ldots\leq a_r$, depending on $p$, are such that $G_p=\mathbb Z_{p^{a_1}}\times\ldots\times\mathbb Z_{p^{a_r}}$ is the $p$-component of $G$, and $[a]_q=1+q+q^2+\ldots+q^{a-1}$.

Finally, we will comment on the general question on whether the associated quantum algebraic objects, like the subfactor, planar algebra, or quantum permutation group, ``see'' or not the defect. Our remark here is that for a finite abelian group $G$ we have: 
$$d(F_G)=N^2\lim_{k\to\infty}\lim_{l\to\infty}\frac{1}{l}\sum_{r=1}^l\int_G\chi_r^k$$

Here the quantities $\chi_r$ on the right are Diaconis-Shahshahani type variables \cite{dsh} associated to the embedding $G\subset S_N$, with $N=|G|$. Now since for an arbitrary quantum permutation group $G\subset S_N^+$ these variables make sense, cf. \cite{bcs}, the above formula makes sense as well, and we will raise here the question regarding its possible validity.

The paper is organized as follows: 1 is a preliminary section, in 2-3 we review the notion of defect, from the above-mentioned algebraic geometric point of view, in 4-5 we investigate the Fourier matrices, and 6 contains the probabilistic speculations.

\section{Complex Hadamard matrices}

We consider in this paper various square matrices over the complex numbers, $H\in M_N(\mathbb C)$. The indices of our matrices will usually range in the set $\{0,1,\ldots,N-1\}$.

\begin{definition}
A complex Hadamard matrix is a matrix $H\in M_N(\mathbb C)$ whose entries are on the unit circle, $|H_{ij}|=1$, and whose rows are pairwise orthogonal.
\end{definition}

 It follows from definitions that the columns of $H$ are pairwise orthogonal as well. 

The main example is the Fourier matrix, based on the root of unity $w=e^{2\pi i/N}$:
$$F_N=\begin{pmatrix}
1&1&1&\ldots&1\\
1&w&w^2&\ldots&w^{N-1}\\
\ldots&\ldots&\ldots&\ldots&\ldots\\
1&w^{N-1}&w^{2(N-1)}&\ldots&w^{(N-1)^2}
\end{pmatrix}$$

There are many other complex Hadamard matrices, see \cite{tz1}. One simple way of constructing new examples is by taking tensor products, as for instance:
$$F_2\otimes F_2=\begin{pmatrix}
1&1\\
1&-1
\end{pmatrix}
\otimes
\begin{pmatrix}
1&1\\
1&-1
\end{pmatrix}
=\begin{pmatrix}
1&1&1&1\\
1&-1&1&-1\\
1&1&-1&-1\\ 
1&-1&-1&1
\end{pmatrix}$$

Observe that $F_N$ is nothing but the matrix of the discrete Fourier transform, over the cyclic group $\mathbb Z_N$. More generally, we have the following construction:

\begin{proposition}
The Fourier matrix of a finite abelian group $G=\mathbb Z_{N_1}\times\ldots\times\mathbb Z_{N_r}$ is the complex Hadamard matrix $F_G=F_{N_1}\otimes\ldots\otimes F_{N_r}$.
\end{proposition}

\begin{proof}
For a product of groups $G=G'\times G''$ we have $F_G=F_{G'}\otimes F_{G''}$, and together with the above observation regarding $F_N$, this gives the equality in the statement.
\end{proof}

As an example, for the Klein group $\mathbb Z_2\times\mathbb Z_2$ we obtain the above matrix $F_2\otimes F_2$.

Now back to the general case, observe that a matrix remains Hadamard when permuting rows and columns, or multiplying them by complex numbers of modulus 1. This is of course quite a trivial observation, and it is convenient to use the following notion:

\begin{definition}
Two matrices are called ``equivalent'' if one can pass from one to the other by permuting rows and columns, or multiplying them by complex numbers of modulus 1. 
\end{definition}

Observe that each Hadamard matrix can be supposed, up to equivalence, to be in ``dephased'' form, in the sense that the first row and column consist of 1 entries only. Note that the Fourier matrices $F_G=F_{N_1}\otimes\ldots\otimes F_{N_r}$ are by definition dephased.

The Fourier matrices are of course quite well understood. The following general construction, inspired from \cite{dit}, allows one, starting for instance from two Fourier matrices, to construct some quite non-trivial examples of complex Hadamard matrices:

\begin{definition}
The deformation of a tensor product $H\otimes K\in M_{NM}(\mathbb C)$, with matrix of parameters $L\in M_{M\times N}(\mathbb T)$, is $H\otimes_LK=(H_{ij}L_{aj}K_{ab})_{ia,jb}$.
\end{definition}

Observe that for the ``flat'' matrix of parameters, $L_{aj}=1$, we obtain the usual tensor product $H\otimes K$. Observe also that in the above definition we can always assume $L$ to be ``dephased'', in the sense that its first row and column consist of 1 entries only.

As a first example, by deforming the tensor product $F_2\otimes F_2$, with the matrix of parameters $L=(^1_1{\ }^1_q)$, we obtain the following complex Hadamard matrix:
$$F_{2,2}^q=\begin{pmatrix}
1&1\\
1&-1
\end{pmatrix}
\otimes_{\begin{pmatrix}
1&1\\
1&q
\end{pmatrix}}
\begin{pmatrix}
1&1\\
1&-1
\end{pmatrix}
=\begin{pmatrix}
1&1&1&1\\
1&-1&q&-q\\
1&1&-1&-1\\ 
1&-1&-q&q
\end{pmatrix}$$

This matrix is in fact the only one at $N=4$, so let us take a closer look at it. As a first observation, at $q\in\{1,i,-1,-i\}$ we obtain tensor products of Fourier matrices:

\begin{proposition}
The matrices $F_{2,2}^q$ are as follows:
\begin{enumerate}
\item At $q=1$ we have $F_{2,2}^q=F_2\otimes F_2$.

\item At $q=-1$ we have $F_{2,2}^q\simeq F_2\otimes F_2$.

\item At $q=\pm i$ we have $F_{2,2}^q\simeq F_4$.
\end{enumerate}
\end{proposition}

\begin{proof}
The first assertion is clear, and the second one follows from it, by permuting the third and the fourth columns:
$$F_{2,2}^{-1}=\begin{pmatrix}
1&1&1&1\\
1&-1&-1&1\\
1&1&-1&-1\\ 
1&-1&1&-1
\end{pmatrix}\simeq\begin{pmatrix}
1&1&1&1\\
1&-1&1&-1\\
1&1&-1&-1\\ 
1&-1&-1&1
\end{pmatrix}=F_{2,2}^1$$

As for the third assertion, this follows from the following computation:
$$F_{2,2}^{\pm i}=\begin{pmatrix}
1&1&1&1\\
1&-1&\pm i&\mp i\\
1&1&-1&-1\\ 
1&-1&\mp i&\pm i
\end{pmatrix}\simeq
\begin{pmatrix}
1&1&1&1\\
1&i&-1&-i\\
1&-1&1&-1\\ 
1&-i&-1&i
\end{pmatrix}=F_4$$

Here we have interchanged the second column with the third one in the case $q=i$, and we have used a cyclic permutation of the last 3 columns in the case $q=-i$. 
\end{proof}

As a second example, at $N=6$ we have two possible deformations of the Fourier matrix $F_6=F_2\otimes F_3=F_3\otimes F_2$. In dephased form, these matrices are given by:
$$F_{2,3}^{(r,s)}=F_2
\otimes_{\begin{pmatrix}
1&1\\
1&r\\
1&s
\end{pmatrix}}
F_3\,,\quad\quad\quad
F_{3,2}^{(r,s)}=F_3
\otimes_{\begin{pmatrix}
1&1&1\\
1&r&s
\end{pmatrix}}
F_2$$

In terms of the parameter $q=(r,s)\in\mathbb T^2$, and with $j=e^{2\pi i/3}$, we have:
$$F_{2,3}^q=\begin{pmatrix}
1&1&1&1&1&1\\
1&j&j^2&r&jr&j^2r\\
1&j^2&j&s&j^2s&js\\ 
1&1&1&-1&-1&-1\\
1&j&j^2&-r&-jr&-j^2r\\
1&j^2&j&-s&-j^2s&-js
\end{pmatrix},\quad
F_{3,2}^q
=\begin{pmatrix}
1&1&1&1&1&1\\
1&-1&r&-r&s&-s\\
1&1&j&j&j^2&j^2\\ 
1&-1&jr&-jr&j^2s&-j^2s\\
1&1&j^2&j^2&j&j\\
1&-1&j^2r&-j^2r&js&-js
\end{pmatrix}$$

There are many other examples of Hadamard matrices at $N=6$, see \cite{tz1}. Of particular interest are the following two matrices, due to Haagerup and Tao \cite{ha1}, \cite{tao}:
$$H_6^q=\begin{pmatrix}
1&1&1&1&1&1\\
1&-1&i&i&-i&-i\\ 
1&i&-1&-i&q&-q\\ 
1&i&-i&-1&-q&q\\
1&-i&\bar{q}&-\bar{q}&i&-1\\ 
1&-i&-\bar{q}&\bar{q}&-1&i
\end{pmatrix},\quad 
T_6=\begin{pmatrix}
1&1&1&1&1&1\\ 
1&1&j&j&j^2&j^2\\ 
1&j&1&j^2&j^2&j\\
1&j&j^2&1&j&j^2\\ 
1&j^2&j^2&j&1&j\\ 
1&j^2&j&j^2&j&1
\end{pmatrix}$$

The point with these matrices is that they are ``regular'', in the sense that any scalar product between distinct rows decomposes as a sum of quantities of type $\sum_{k=1}^n\lambda w^k$, with $\lambda\in\mathbb T$ and $w=e^{2\pi i/n}$. Observe that the deformed Fourier matrices are regular too.

The complex Hadamard matrices of size $N\leq 5$ were classified by Haagerup in \cite{ha1} and the regular matrices at $N=6$ were classified in \cite{bbs}, and we have:

\begin{theorem}
The complex Hadamard matrices of small order are as follows:
\begin{enumerate}
\item $F_2,F_3,F_{2,2}^q,F_5$ are the only examples at $N=2,3,4,5$.

\item $F_{2,3}^q,F_{3,2}^q,H_6^q,T_6$ are the only regular examples at $N=6$.
\end{enumerate}
\end{theorem}

\begin{proof}
The idea is that, the rows of our matrix being pairwise orthogonal, we must first understand the solutions of $a_1+\ldots+a_N=0$, with $|a_1|=\ldots=|a_N|=1$. 

At $N=2,3,4$ this is very simple: our equation must be, up to a permutation of the terms, a ``trivial'' equation of type $a-a=0$, $a+ja+j^2a=0$ or $a-a+b-b=0$ respectively, with $|a|=|b|=1$ and $j=e^{2\pi i/3}$. Now by using this observation, and trying to build a complex Hadamard matrix of size $N$, this leads to $F_2,F_3,F_{2,2}^q$ only.

At $N=5,6$ understanding the solutions of the above equation is a particularly difficult task. However, Haagerup was able to obtain the above result (1), see \cite{ha1}. As for (2), this result is from \cite{bbs}, with the proof this time being also long, but purely combinatorial.
\end{proof}

\section{Enveloping tangent spaces}

We study now the real algebraic manifold formed by all the $N\times N$ complex Hadamard matrices. Let us begin with some definitions:

\begin{definition}
Let $M_N'(\mathbb C)\subset M_N(\mathbb C)$ be the set of matrices having $1$ on the first row and column, and make $S_{N-1}\times S_{N-1}$ act on it by permuting rows and columns. We set:
\begin{enumerate}
\item $C_N=M_N(\mathbb T)\cap\sqrt{N}U_N$: the manifold of $N\times N$ complex Hadamard matrices.

\item $D_N=C_N\cap M_N'(\mathbb C)$: the submanifold consisting of dephased matrices.

\item $E_N=D_N/(S_{N-1}\times S_{N-1})$: the equivalence classes of matrices in $C_N$.
\end{enumerate}
\end{definition}

Here, and in what follows, by ``manifold'' we will always mean ``real algebraic manifold''. Observe that $C_N,D_N$ are indeed real algebraic manifolds, as being by definition intersections of such manifolds. As for $E_N$, as defined above, this is just a set.

Observe that we have surjective maps, as follows:
$$C_N\to D_N\to E_N$$

Here the first map is by definition obtained by ``dephasing'' the matrix, in the obvious way, and the second map is the canonical quotient map. Note that the dephasing map $C_N\to D_N$ is continuous, and is a covering of real algebraic manifolds.

Let us first record a result coming from the classification results above:

\begin{proposition}
The sets $E_N$ with $N$ small are as follows:
\begin{enumerate}
\item $E_2=\{F_2\}$, $E_3=\{F_3\}$, $E_5=\{F_5\}$.

\item $E_4=\{F_{2,2}^q|q\in\mathbb T\}$.

\item $E_6=\{F_{2,3}^q|q\in\mathbb T^2\}\cup\{F_{3,2}^q|q\in\mathbb T^2\}\cup\{H_6^q|q\in\mathbb T\}\cup\{T_6\}\cup X$.
\end{enumerate}
\end{proposition}

\begin{proof}
This is just a reformulation of Theorem 1.6 above, with $X$ being by definition the set of equivalence classes of non-regular $6\times6$ complex Hadamard matrices.
\end{proof}

With this result in hand, we can now formulate a few basic observations:

\begin{proposition}
The manifolds $C_N,D_N$ are in general not smooth, and not connected. Nor they are in general complex algebraic manifolds.
\end{proposition}

\begin{proof}
All the assertions basically follow from Proposition 2.2, by using the canonical surjective maps $C_N\to D_N\to E_N$ in order to get back to $D_N$ and $C_N$. 

Indeed, from $|E_2|=1$ and from $E_4\simeq\mathbb T$ we obtain that $C_2,D_4$ have real dimensions 3 and 1, so they cannot have a complex algebraic manifold structure.

Regarding now the smoothness claim, let us first look at $D_4$. This manifold consists of $|S_3\times S_3|=36$ copies of $\{F_{2,2}^q|q\in\mathbb T\}\simeq\mathbb T$, that we will denote as follows:
$$\mathbb T^{(\pi,\sigma)}=\{(\pi,\sigma)F_{2,2}^q|q\in\mathbb T\}$$

Here the element $(\pi,\sigma)\in S_3\times S_3$ acts as usual on $M_4'(\mathbb C)$, with $\pi$ permuting the rows labeled 1,2,3, and $\sigma$ permuting the columns labeled 1,2,3 (recall that, according to our usual conventions, the first row and column of our matrices are labeled 0). 

The point now is that some of these 36 copies of $\mathbb T$ intersect at $q=1$, which prevents $D_4$ from being smooth. For instance the copy of $\mathbb T$ associated to $((12),(12))$ is: 
$$\mathbb T^{((12),(12))}=\left\{\begin{pmatrix}
1&1&1&1\\
1&-1&1&-1\\
1&q&-1&-q\\
1&-q&-1&q
\end{pmatrix}\Big|q\in\mathbb T\right\}$$

We can see from this formula that $\mathbb T^{((12),(12))}$ is distinct from $\mathbb T^{(id,id)}=\{F_{2,2}^q|q\in\mathbb T\}$, but intersects it at $q=1$, so $D_4$ is indeed not smooth. Similarly, $C_4$ is not smooth either, because it consists of 36 copies of $\mathbb T\times\mathbb T^7\simeq\mathbb T^8$, which intersect non-trivially at $q=1$.

For the connectedness question now, this follows from the fact that $T_6$ and its conjugates are isolated in $D_6$, cf. \cite{nic}, \cite{tz2}. By lifting, $C_6$ cannot be connected either.
\end{proof}

Regarding now the set $E_N$, it is not clear how to give it a structure of real algebraic manifold, or even a nice topological space structure. Observe that we have an embedding $E_N\subset M_N(\mathbb T)$ coming from the lexicographic order on $M_N(\mathbb T)$ induced by the $[0,2\pi)$ order on $\mathbb T$. It is not clear whether this embedding produces a nice space or not. 

In order to get now some insight into the structure of $C_N,D_N$, we will compute the corresponding enveloping tangent spaces, coming from Definition 2.1 above: 

\begin{definition}
The ``enveloping tangent spaces'' of $C_N,D_N$ are defined by:
\begin{enumerate}
\item For $H\in C_N$ we let $\widetilde{T}_HC_N=T_HM_N(\mathbb T)\cap T_H\sqrt{N}U_N$.

\item For $H\in D_N$ we let $\widetilde{T}_HD_N=\widetilde{T}_HC_N\cap T_HM_N'(\mathbb C)$.
\end{enumerate}
\end{definition}

Here, and in what follows, we denote by $T_HM$ the tangent space to a smooth real manifold $M$, at a given point $H\in M$. Observe that, $C_N,D_N$ being not smooth manifolds in general, $\widetilde{T}_HC_N,\widetilde{T}_HD_N$ are of course not their ``tangent spaces'' in the usual sense.

Tadej and \.Zyczkowski computed in \cite{tz2} a certain vector space which coincides with the above enveloping tangent space. Their result, formulated in our terms, is as follows:

\begin{theorem}
We have a canonical identification
$$\widetilde{T}_HC_N=\left\{A\in M_N(\mathbb R)\Big|\sum_kH_{ik}\bar{H}_{jk}(A_{ik}-A_{jk})=0,\,\forall i,j\right\}$$
and $\widetilde{T}_HD_N$ consists of the matrices $A\in\widetilde{T}_HC_N$ having $0$ on the first row and column.
\end{theorem}

\begin{proof}
We use the notation $H_{ij}=X_{ij}+iY_{ij}$. We know that $M_N(\mathbb T)$ is defined by the algebraic relations $|H_{ij}|^2=1$, with $i,j\in\{0,1,\ldots,N-1\}$, and we have:
$$d|H_{ij}|^2=d(X_{ij}^2+Y_{ij}^2)=2(X_{ij}\dot{X}_{ij}+Y_{ij}\dot{Y}_{ij})$$

Consider now an arbitrary vector $v\in T_HM_N(\mathbb C)$, written as:
$$v=\sum_{ij}\alpha_{ij}\dot{X}_{ij}+\beta_{ij}\dot{Y}_{ij}$$

By taking the scalar product with the above quantity, we get:
$$<v,d|H_{ij}|^2>=2(\alpha_{ij}X_{ij}+\beta_{ij}Y_{ij})$$

Now since these scalar products vanish when we have $\alpha_{ij}=A_{ij}Y_{ij}$ and $\beta_{ij}=-A_{ij}X_{ij}$, for certain numbers $A_{ij}\in\mathbb R$, we obtain the following formula: 
$$T_HM_N(\mathbb T)=\left\{\sum_{ij}A_{ij}(Y_{ij}\dot{X}_{ij}-X_{ij}\dot{Y}_{ij})\Big|A_{ij}\in\mathbb R\right\}$$

Let us compute now the subspace in the statement. We know that $\sqrt{N}U_N$ is defined by the algebraic relations $<H_i,H_j>=N\delta_{ij}$, where $H_0,\ldots,H_{N-1}$ are the rows of $H$. Since the relations $<H_i,H_i>=N$ are automatic for matrices $H\in M_N(\mathbb T)$, we can remove them from the computation. So, if for $i\neq j$ we let $L_{ij}=d<H_i,H_j>$, then we have:
$$\widetilde{T}_HC_N=\left\{ v\in T_HM_N(\mathbb T)|<v,L_{ij}>=0,\,\forall i\neq j\right\}$$

So, let us compute the scalar product appearing above. First, we have:
\begin{eqnarray*}
L_{ij}
&=&d<H_i,H_j>=d\left(\sum_kH_{ik}\bar{H}_{jk}\right)=\sum_kH_{ik}\dot{\bar{H}}_{jk}+\bar{H}_{jk}\dot{H}_{ik}\\
&=&\sum_k(X_{ik}+iY_{ik})(\dot{X}_{jk}-i\dot{Y}_{jk})+(X_{jk}-iY_{jk})(\dot{X}_{ik}+i\dot{Y}_{ik})
\end{eqnarray*}

Consider now an arbitrary vector $v\in T_HM_N(\mathbb T)$. According to the above formula of $T_HM_N(\mathbb T)$, we can write this vector in terms of a real matrix $(A_{lk})$, as follows:
$$v=\sum_{lk}A_{lk}(Y_{lk}\dot{X}_{lk}-X_{lk}\dot{Y}_{lk})$$

Thus the above scalar products defining the space $\widetilde{T}_HC_N$ are given by:
\begin{eqnarray*}
<v,L_{ij}>
&=&\sum_k<A_{ik}(Y_{ik}\dot{X}_{ik}-X_{ik}\dot{Y}_{ik})+A_{jk}(Y_{jk}\dot{X}_{jk}-X_{jk}\dot{Y}_{jk}),L_{ij}>\\
&=&\sum_kA_{ik}(Y_{ik}+iX_{ik})(X_{jk}+iY_{jk})+A_{jk}(Y_{jk}-iX_{jk})(X_{ik}-iY_{ik})\\
&=&i\sum_kA_{ik}\bar{H}_{ik}H_{jk}-A_{jk}H_{jk}\bar{H}_{ik}\\
&=&i\sum_k(A_{ik}-A_{jk})\bar{H}_{ik}H_{jk}
\end{eqnarray*}

Thus we have reached to the description of $\widetilde{T}_HC_N$ in the statement. The description of $\widetilde{T}_HD_N$ is the statement follows as well, by intersecting with $T_HM_N'(\mathbb C)$.
\end{proof}

Let us look now at the dimensions of the spaces in Theorem 2.5:

\begin{definition}
Associated to a complex Hadamard matrix $H\in M_N(\mathbb C)$ are:
\begin{enumerate}
\item The dephased defect $d'(H)=\dim(\widetilde{T}_HD_N)$.

\item The undephased defect $d(H)=\dim(\widetilde{T}_HC_N)$.
\end{enumerate}
\end{definition}

Observe that the dephased defect $d'(H)$ is nothing but the quantity ${\rm def}(H)$ from \cite{tz2}, simply called ``defect'' there. In what follows we will rather use the quantity $d(H)$, which behaves better with respect to the various operations on complex Hadamard matrices. 

The dephased and undephased defect are related by the following formula:

\begin{proposition}
$d'(H)=d(H)-2N+1$.
\end{proposition}

\begin{proof}
Consider the vector space $M_N^\circ(\mathbb R)\subset M_N(\mathbb R)$ formed by the matrices of type $A_{ij}=a_i+b_j$, with $a_i,b_j\in\mathbb R$. We claim that we have an isomorphism as follows:
$$\widetilde{T}_HC_N\simeq \widetilde{T}_HD_N\oplus M_N^\circ(\mathbb R)$$

Indeed, the first remark is that for a matrix of type $A_{ij}=a_i+b_j$ we have:
\begin{eqnarray*}
\sum_kH_{ik}\bar{H}_{jk}(A_{ik}-A_{jk})
&=&\sum_kH_{ik}\bar{H}_{jk}(a_i+b_k-a_j-b_k)\\
&=&(a_i-a_j)\sum_kH_{ik}\bar{H}_{jk}\\
&=&(a_i-a_j)\delta_{ij}=0
\end{eqnarray*}

Thus we have $M_N^\circ(\mathbb R)\subset\widetilde{T}_HC_N$. Also, we have $M_N^\circ(\mathbb R)\cap\widetilde{T}_HD_N=\{0\}$, because $a_i+b_j=0$ if $i=0$ or $j=0$ implies $a_i+b_j=0$ for any $i,j$. Finally, let $A\in\widetilde{T}_HC_N$, and  decompose it as $A=B+C$, with $B_{ij}=A_{i0}+A_{0j}-A_{00}$ and $C_{ij}=A_{ij}-B_{ij}$. Then we have $B\in M_N^\circ(\mathbb R)$, so $C\in \widetilde{T}_HC_N$, and since $C_{i0}=C_{0j}=0$ for any $i,j$, we conclude that $C\in\widetilde{T}_HD_N$.

Summing up, we have proved the above vector space claim. Now since we have $\dim(M_N^\circ(\mathbb R))=2N-1$, we obtain the formula in the statement.
\end{proof}

The study of the defect is motivated by the following simple fact:

\begin{proposition}
If $d'(H)=0$ then $H$  is isolated inside $D_N$.
\end{proposition}

\begin{proof}
Indeed, if the matrix $H\in D_N$ was not isolated, then the tangent vector to any one-parameter family $H^q$ would belong to the space $\widetilde{T}_HD_N=\emptyset$, contradiction.
\end{proof}

In general, it is not very clear what exact local information about $C_N$ is encoded by the defect. For a detailed discussion here, we refer to the recent article \cite{bbe}.

\section{Examples and comments}

We discuss in this section the various particular cases of Theorem 2.5. Let us first examine the product operations. We have here the following statement:

\begin{proposition}
The enveloping tangent spaces are as follows:
\begin{enumerate}
\item For equivalent matrices $H\sim H'$ we have $\widetilde{T}\simeq\widetilde{T}'$.

\item For tensor products $H=H'\otimes H''$ we have $\widetilde{T}'\otimes\widetilde{T}''\subset\widetilde{T}$.
\end{enumerate}
\end{proposition}

\begin{proof}
These results, from \cite{tad}, \cite{tz2}, follow as well from Theorem 2.5:

(1) This is clear from definitions, because the permutations and the multiplication by scalars act as well, and in a compatible way, on the enveloping tangent spaces.

(2) Indeed, for a matrix $A=A'\otimes A''$ with $A'\in\widetilde{T}'$ and $A''\in\widetilde{T}''$, we have:
\begin{eqnarray*}
\sum_{kc}H_{ia,kc}\bar{H}_{jb,kc}A_{ia,kc}
&=&\sum_{kc}H'_{ik}H''_{ac}\bar{H}'_{jk}\bar{H}''_{bc}A'_{ik}A''_{ac}\\
&=&\sum_kH'_{ik}\bar{H}'_{jk}A'_{ik}\sum_cH''_{ac}\bar{H}''_{bc}A''_{ac}\\
&=&\sum_kH'_{ik}\bar{H}'_{jk}A'_{jk}\sum_cH''_{ac}\bar{H}''_{bc}A''_{bc}\\
&=&\sum_{kc}H'_{ik}H''_{ac}\bar{H}'_{jk}\bar{H}''_{bc}A'_{jk}A''_{bc}\\
&=&\sum_{kc}H_{ia,kc}\bar{H}_{jb,kc}A_{jb,kc}
\end{eqnarray*}

Thus we have indeed $A\in\widetilde{T}$, and we are done.
\end{proof}

In terms of the defect, we obtain:

\begin{proposition}
The undephased defect satisfies:
\begin{enumerate}
\item For equivalent matrices $H\sim H'$ we have $d(H)=d(H')$.

\item For tensor products $H=H'\otimes H''$ we have $d(H)\geq d(H')d(H'')$.
\end{enumerate}
\end{proposition}

\begin{proof}
These results, once again from \cite{tad}, \cite{tz2}, follow as well from Proposition 3.1.
\end{proof}

Observe that we don't have equality in the tensor product estimate, even in very simple cases. For instance if we consider two Fourier matrices $F_2$, we obtain:
$$d(F_2\otimes F_2)=10>9=d(F_2)^2$$

Here the number $10$ comes from the general formula of $d(F_G)$ explained in section 5 below, in the particular case $G=\mathbb Z_2\times\mathbb Z_2$, or from Proposition 3.5 below, at $q=1$.

The problem of finding upper bounds for the defect of a tensor product was investigated in detail by Tadej in \cite{tad}. We would like to raise here the following related question:

\begin{problem}
How does the defect of a deformed tensor product, $d(H\otimes_LK)$, vary with the matrix of parameters $L\in M_{M\times N}(\mathbb T)$?
\end{problem}

This problem looks quite complicated, even at small values of $M,N$. As a first observation, the equations of the enveloping tangent space are:
$$\sum_{kc}L_{ak}\bar{L}_{bk}H_{ik}\bar{H}_{jk}K_{ac}\bar{K}_{bc}(A_{ia,kc}-A_{jb,kc})=0$$

There is no obvious trick that can be applied here. Note that there is no reason for a ``transport formula'' of type $d(H\otimes_LK)=d(H\otimes K)$ to hold, even in simple cases. Indeed, the $L$-deformation procedure ``deforms well'' the matrix $H\otimes K$, but maybe not the other $NM\times NM$ complex Hadamard matrices, which might happen to be around. 

However, there are several special cases where the problem can be solved, for instance by combining the formulae for Fourier matrices with the following observation:

\begin{proposition}
We have $F_{NM}\simeq F_N\otimes_LF_M$, where $L_{aj}=w^{aj}$, with $w=e^{2\pi i/NM}$.
\end{proposition}

\begin{proof}
Indeed, by using $w^{NM}=1$, we obtain the following formula:
\begin{eqnarray*}
(F_N\otimes_LF_M)_{ia,jb}
&=&(F_N)_{ij}L_{aj}(F_M)_{ab}\\
&=&w^{Mij+aj+Nab}\\
&=&w^{(Ni+a)(j+Nb)}\\
&=&(F'_{NM})_{ia,jb}
\end{eqnarray*}

Here $F'_{NM}$ is a certain matrix which is equivalent to $F_{NM}$, and we are done.
\end{proof}

Let us discuss now the simplest case of the problem, $N=M=2$. We will work out everything in detail, as an illustration for how the equations in Theorem 2.5 work.

\begin{proposition}
We have $d(F_{2,2}^q)=10$ at $q=\pm1$, and $d(F_{2,2}^q)=8$ at $q\neq\pm1$.
\end{proposition}

\begin{proof}
Our starting point are the equations in Theorem 2.5, namely:
$$\sum_hH_{ik}\bar{H}_{jk}(A_{ik}-A_{jk})=0$$

Since the $i>j$ equations are equivalent to the $i<j$ ones, and the $i=j$ equations are trivial, we just have to write down the equations corresponding to indices $i<j$. And, with $ij=01,02,03,12,13,23$, these equations are:
\begin{eqnarray*}
(A_{00}-A_{10})-(A_{01}-A_{11})+\bar{q}(A_{02}-A_{12})-\bar{q}(A_{03}-A_{13})&=&0\\
(A_{00}-A_{20})+(A_{01}-A_{21})-(A_{02}-A_{22})-(A_{03}-A_{23})&=&0\\
(A_{00}-A_{30})-(A_{01}-A_{31})-\bar{q}(A_{02}-A_{32})+\bar{q}(A_{03}-A_{33})&=&0\\
(A_{10}-A_{20})-(A_{11}-A_{21})-q(A_{12}-A_{22})+q(A_{13}-A_{23})&=&0\\
(A_{10}-A_{30})+(A_{11}-A_{31})-(A_{12}-A_{32})-(A_{13}-A_{33})&=&0\\
(A_{20}-A_{30})-(A_{21}-A_{31})+\bar{q}(A_{22}-A_{32})-\bar{q}(A_{23}-A_{33})&=&0
\end{eqnarray*}

Assume first $q\neq\pm 1$. Then $q$ is not real, and appears in 4 of the above equations. But these 4 equations can be written in the following way:
\begin{eqnarray*}
(A_{00}-A_{01})-(A_{10}-A_{11})+\bar{q}((A_{02}-A_{03})-(A_{12}-A_{13}))&=&0\\
(A_{00}-A_{01})-(A_{30}-A_{31})-\bar{q}((A_{02}-A_{03})-(A_{32}-A_{33}))&=&0\\
(A_{10}-A_{11})-(A_{20}-A_{21})-q((A_{12}-A_{13})-(A_{22}-A_{23}))&=&0\\
(A_{20}-A_{21})-(A_{30}-A_{31})+\bar{q}((A_{22}-A_{23})-(A_{32}-A_{33}))&=&0
\end{eqnarray*}

Now since the unknowns are real, and $q$ is not, we conclude that the terms between braces in the left part must be all equal, and that the same must happen at right:
\begin{eqnarray*}
A_{00}-A_{01}&=&A_{10}-A_{11}=A_{20}-A_{21}=A_{30}-A_{31}\\
A_{02}-A_{03}&=&A_{12}-A_{13}=A_{22}-A_{23}=A_{32}-A_{33}
\end{eqnarray*}

Thus, the equations involving $q$ tell us that $A$ must be of the following form:
$$A=\begin{pmatrix}
a&a+x&e+y&e\\
b&b+x&f+y&f\\
c&c+x&g+y&g\\
d&d+x&h+y&h\end{pmatrix}$$

Let us plug now these values in the remaining 2 equations. We obtain:
\begin{eqnarray*}
a-c+a+x-c-x-e-y+g+y-e+g&=&0\\
b-d+b+x-d-x-f-y+h+y-f+h&=&0
\end{eqnarray*}

Thus we must have $a+g=c+e$ and $b+h=d+f$, which are independent conditions. We conclude that the dimension of the space of solutions is $10-2=8$, as claimed.

Assume now $q=\pm 1$. For simplicity we set $q=1$, by using  Proposition 3.1, and we use as well Proposition 2.7, which tells us that it is enough to compute the dephased defect. The dephased equations, obtained by setting $A_{i0}=A_{0j}=0$ in our system, are:
\begin{eqnarray*}
A_{11}-A_{12}+A_{13}&=&0\\
-A_{21}+A_{22}+A_{23}&=&0\\
A_{31}+A_{32}-A_{33}&=&0\\
-A_{11}+A_{21}-A_{12}+A_{22}+A_{13}-A_{23}&=&0\\
A_{11}-A_{31}-A_{12}+A_{32}-A_{13}+A_{33}&=&0\\
-A_{21}+A_{31}+A_{22}-A_{32}-A_{23}+A_{33}&=&0
\end{eqnarray*}

The first three equations tell us that our matrix must be of the following form:
$$A=\begin{pmatrix}a&a+b&b\\ c+d&c&d\\ e&f&e+f\end{pmatrix}$$

Now by plugging these values in the last three equations, these become:
\begin{eqnarray*}
-a+c+d-a-b+c+b-d&=&0\\
a-e-a-b+f-b+e+f&=&0\\
-c-d+e+c-f-d+e+f&=&0
\end{eqnarray*}

Thus we must have $a=c$, $b=f$, $d=e$, and since these conditions are independent, the dephased defect is 3, and so the undephased defect is $3+7=10$, as claimed.
\end{proof}

Let us discuss now the reformulation of Theorem 2.5, for certain special classes of complex Hadamard matrices. The simplest situation is that when we have a usual Hadamard matrix, $H\in M_N(\pm 1)$. The combinatorics of $H$ is encoded into the array $\varepsilon\in M_{N\times N\times N}(\pm 1)$ given by $\varepsilon_{ijk}=H_{ik}H_{jk}$, that we will call ``design'' of $H$. See \cite{aga}, \cite{hor}.

\begin{proposition}
In the real case $H\in M_N(\pm 1)$ the system of equations for $\widetilde{T}$ is
$$\sum_k\varepsilon_{ijk}(A_{ik}-A_{jk})=0$$
where $\varepsilon\in M_{N\times N\times N}(\pm 1)$, given by $\varepsilon_{ijk}=H_{ik}H_{jk}$, is the ``design array'' of $H$.
\end{proposition}

\begin{proof}
This is just an observation, clear from definitions.
\end{proof}

It is not clear what the spectral properties of $\varepsilon$ are, and how the above formulation can be improved. Let us also mention that of particular interest would be an extension of the above observation to the root of unity case \cite{but}. However, as shown by Lam and Leung in \cite{lle}, the ``design'' of a Butson matrix can be something of a very complicated nature, unless the ``regularity conjecture'' in \cite{bbs} holds indeed. One way of overcoming these problems would be by looking directly at the matrices which are regular in the sense of \cite{bbs}. But the corresponding ``design matrices'' are not axiomatized yet.

Let us discuss as well the circulant case. Here our matrices must be ``Fourier-diagonal'', of type $H=\sqrt{N}FQF^*$, with $F=F_N/\sqrt{N}$ and with $Q$ diagonal over $\mathbb T$. The problem of finding the diagonal vectors $Q\in\mathbb T^N$ having the property that $H=\sqrt{N}FQF^*$ is indeed Hadamard is a subtle Fourier analysis problem, cf. \cite{bnz}, \cite{bjo}, \cite{bfr}, \cite{ha2}.

In Fourier notation, our system of equations is as follows:

\begin{proposition}
In the circulant case the equations for $\widetilde{T}$ are
$$\sum_{klr}w^{k(l-r)-il+jr}Q_l\bar{Q}_r(A_{ik}-A_{jk})=0$$
where $Q\in\mathbb T^N$ is the eigenvalue vector of $U=H/\sqrt{N}$.
\end{proposition}

\begin{proof}
First, with $H_{ij}=C_{j-i}$, the system of equations for $\widetilde{T}$ becomes:
$$\sum_kC_{k-i}\bar{C}_{k-j}(A_{ik}-A_{jk})=0$$

Now since we have $C=FQ$, where $Q\in\mathbb T^N$ is the eigenvalue vector, we get:
$$\sum_{klr}w^{(k-i)l}Q_lw^{-(k-j)r}\bar{Q}_r(A_{ik}-A_{jk})=0$$

By rearranging the terms, this gives the formula in the statement.
\end{proof}

Once again, there is no obvious trick that can be applied to the equations. The circulant problem seems to require a delicate case-by-case analysis, and we have no results.

We would like to end this section with a few theoretical observations, regarding the general case. First, we have the following linear algebra interpretation of the defect:

\begin{proposition}
The undephased defect $d(H)$ is the corank of the matrix
$$Y_{ij,ab}
=(\delta_{ia}-\delta_{ja})\begin{cases}
Re(H_{ib}\bar{H}_{jb})&i<j\\ 
Im(H_{ib}\bar{H}_{jb})&i>j\\
*&i=j
\end{cases}$$
where $*$ can be any quantity (its coefficient being $0$ anyway).
\end{proposition}

\begin{proof}
The matrix of the system defining the enveloping tangent space is:
$$X_{ij,ab}=(\delta_{ia}-\delta_{ja})H_{ib}\bar{H}_{jb}$$

However, since we are only looking for real solutions $A\in M_N(\mathbb R)$, we have to take into account the real and imaginary parts. But this is not a problem, because the $(ij)$ equation coincides with the $(ji)$ one, that we can cut. More precisely, if we set $Y$ as above, then we obtain precisely the original system. Thus the defect of $H$ is indeed the corank of $Y$. 
\end{proof}

As an illustration, for the Fourier matrix $F_N$ we have the following formula, where $e(i,j)\in\{-1,0,1\}$ is negative if $i<j$, null for $i=j$, and positive for $i>j$:
$$Y_{ij,ab}=\frac{1}{2}(\delta_{ia}-\delta_{ja})(w^{(i-j)b}+e(i,j)w^{(j-i)b})$$

Observe in particular that for the Fourier matrix $F_2$ we have:
$$Y=\begin{pmatrix}0&0&0&0\\ 1&-1&-1&1\\ -1&1&1&-1\\ 0&0&0&0\end{pmatrix}$$

Here the corank is $3$, but, unfortunately, this cannot be seen on the characteristic polynomial, which is $P(\lambda)=\lambda^4$. The problem is that our matrix, and more precisely its middle $2\times 2$ block, is not diagonalizable. This phenomenon seems to hold in general.

Finally, we have as well the following key question:

\begin{question}
Does the profile matrix of $H$, namely
$$M_{ia}^{jb}=\sum_kH_{ik}\bar{H}_{jk}\bar{H}_{ak}H_{bk}$$
determine the enveloping tangent space $\widetilde{T}_HC_N$?
\end{question}

The point here is that the profile matrix $M\in M_{N^2}(\mathbb C)$ is the one which produces the subfactor, planar algebra, or quantum permutation group associated to $H\in M_N(\mathbb T)$. Thus, the above question is of great importance for understanding the relation between the ``geometric'' and ``quantum'' invariants of the complex Hadamard matrices.

An idea for parametrizing the system in terms of the profile matrix $M$ might come from the solution for the Fourier matrix, explained in section 4 below. Indeed, if we write $P=AH/N$, so that $A=PH^*$, then our equations become:
$$\sum_k\bar{H}_{ks}H_{ik}\bar{H}_{jk}\sum_s(P_{is}-P_{js})=0$$
$$\sum_s\bar{H}_{js}P_{is}=\sum_sH_{j,-s}\bar{P}_{i,-s}$$

Observe the similarity between the matrix on the top left and $M$. However, this does not solve our problem. We will be back to our problem in section 6 below.

\section{Abstract Fourier matrices}

We recall that the tensor products of Fourier matrices are exactly the Fourier matrices of the finite abelian groups. More precisely, with $G=\mathbb Z_{N_1}\times\ldots\times\mathbb Z_{N_r}$, we have:
$$F_G=F_{N_1}\otimes\ldots\otimes F_{N_r}$$

The defect of $F_N$ was computed by Tadej and \.Zyczkowski in \cite{tz2}. The first part of their proof works in fact for abstract Fourier matrices as well, and we have:

\begin{proposition}
For a Fourier matrix $F=F_G$, the matrices $A\in\widetilde{T}_FC_N$, with $N=|G|$, are those of the form $A=PF^*$, with $P\in M_N(\mathbb C)$ satisfying
$$P_{ij}=P_{i+j,j}=\bar{P}_{i,-j}$$
where the indices $i,j$ are by definition taken in the group $G$.
\end{proposition}

\begin{proof}
We use the system of equations in Theorem 2.5, namely:
$$\sum_kF_{ik}\bar{F}_{jk}(A_{ik}-A_{jk})=0$$

Now assume $F=F_{N_1}\otimes\ldots\otimes F_{N_r}$, so that with $w_k=e^{2\pi i/k}$ we have:
$$F_{i_1\ldots i_r,j_1\ldots j_r}=(w_{N_1})^{i_1j_1}\ldots (w_{N_r})^{i_rj_r}$$

With $N=N_1\ldots N_r$ and $w=e^{2\pi i/N}$, we obtain:
$$F_{i_1\ldots i_r,j_1\ldots j_r}=w^{\left(\frac{i_1j_1}{N_1}+\ldots+\frac{i_rj_r}{N_r}\right)N}$$

Thus the matrix of our system is given by:
$$F_{i_1\ldots i_r,k_1\ldots k_r}\bar{F}_{j_1\ldots j_r,k_1\ldots k_r}=w^{\left(\frac{(i_1-j_1)k_1}{N_1}+\ldots+\frac{(i_r-j_r)k_r}{N_r}\right)N}$$

Now by plugging in a multi-indexed matrix $A$, our system becomes:
$$\sum_{k_1\ldots k_r}w^{\left(\frac{(i_1-j_1)k_1}{N_1}+\ldots+\frac{(i_r-j_r)k_r}{N_r}\right)N}(A_{i_1\ldots i_r,k_1\ldots k_r}-A_{j_1\ldots j_r,k_1\ldots k_r})=0$$

Now observe that in the above formula we have in fact two matrix multipliations, so our system can be simply written as:
$$(AF)_{i_1\ldots i_r,i_1-j_1\ldots i_r-j_r}-(AF)_{j_1\ldots j_r,i_1-j_1\ldots i_r-j_r}=0$$

Now recall that our indices always have a ``cyclic'' meaning, so they belong in fact to the group $G$. So, with $P=AF$, and by using multi-indices, our system is simply:
$$P_{i,i-j}=P_{j,i-j}$$

With $i=I+J,j=I$ we obtain the condition $P_{I+J,J}=P_{IJ}$ in the statement. 

In addition, $A=PF^*$ must be a real matrix. But, if we set $\tilde{P}_{ij}=\bar{P}_{i,-j}$, we have:
\begin{eqnarray*}
\overline{(PF^*)}_{i_1\ldots i_r,j_1\ldots j_r}
&=&\sum_{k_1\ldots k_r}\bar{P}_{i_1\ldots i_r,k_1\ldots k_r}F_{j_1\ldots j_r,k_1\ldots k_r}\\
&=&\sum_{k_1\ldots k_r}\tilde{P}_{i_1\ldots i_r,-k_1\ldots -k_r}(F^*)_{-k_1\ldots -k_r,j_1\ldots j_r}\\
&=&(\tilde{P}F^*)_{i_1\ldots i_r,j_1\ldots j_r}
\end{eqnarray*}

Thus we have $\overline{PF^*}=\tilde{P}F^*$, so the fact that the matrix $PF^*$ is real, which means by definition that we have $\overline{PF^*}=PF^*$, can be reformulated as $\tilde{P}F^*=PF^*$, and hence as $\tilde{P}=P$. So, we obtain the conditions $P_{ij}=\bar{P}_{i,-j}$ in the statement.
\end{proof}

\begin{theorem}
The undephased defect of an abstract Fourier matrix $F_G$ is given by:
$$d(F_G)=\sum_{g\in G}\frac{|G|}{ord(g)}$$
\end{theorem}

\begin{proof}
According to Proposition 4.1 above, the undephased defect $d(F_G)$ is the dimension of the real vector space formed by the matrices $P\in M_N(\mathbb C)$ satisfying:
$$P_{ij}=P_{i+j,j}=\bar{P}_{i,-j}$$

Here, and in what follows, the various indices $i,j,\ldots$ will be taken in $G$. Now the point is that, in terms of the columns of our matrix $P$, the above conditions are:

(1) The entries of the $j$-th column of $P$, say $C$, must satisfy $C_i=C_{i+j}$.

(2) The $(-j)$-th column of $P$ must be conjugate to the $j$-th column of $P$.

Thus, in order to count the above matrices $P$, we can basically fill the columns one by one, by taking into account the above conditions. In order to do so, consider the subgroup $G_2=\{j\in G|2j=0\}$, and then write $G$ as a disjoint union, as follows:
$$G=G_2\sqcup X\sqcup(-X)$$ 

With this notation, the algorithm is as follows. First, for any $j\in G_2$ we must fill the $j$-th column of $P$ with real numbers, according to the periodicity rule $C_i=C_{i+j}$. Then, for any $j\in X$ we must fill the $j$-th column of $P$ with complex numbers, according to the same periodicity rule $C_i=C_{i+j}$. And finally, once this is done, for any $j\in X$ we just have to set the $(-j)$-th column of $P$ to be the conjugate of the $j$-th column.

So, let us compute the number of choices for filling these columns. Our claim is that, when uniformly distributing the choices for the $j$-th and $(-j)$-th columns, for $j\notin G_2$, there are exactly $[G:<j>]$ choices for the $j$-th column, for any $j$. Indeed:

(1) For the $j$-th column with $j\in G_2$ we must simply pick $N$ real numbers subject to the condition $C_i=C_{i+j}$ for any $i$, so we have indeed $[G:<j>]$ such choices.

(2) For filling the $j$-th and $(-j)$-th column, with $j\notin G_2$, we must pick $N$ complex numbers subject to the condition $C_i=C_{i+j}$ for any $i$. Now since there are $[G:<j>]$ choices for these numbers, so a total of $2[G:<j>]$ choices for their real and imaginary parts, on average over $j,-j$ we have $[G:<j>]$ choices, and we are done again.

Summarizing, the dimension of the vector space formed by the matrices $P$, which is equal to the number of choices for the real and imaginary parts of the entries of $P$, is:
$$d(F_G)=\sum_{j\in G}[G:<j>]$$

But this number is exactly the one in the statement, and this finishes the proof.
\end{proof}

The above results suggest a number of concrete and abstract computations, to be done in the next sections. For the moment, let us just record the following conjecture:

\begin{conjecture}
For a regular Hadamard matrix $H\in M_N(\mathbb C)$ we have
$$\widetilde{T}_HC_N=\mathbb C\cdot[\widetilde{T}_HC_N]_{\mathbb Q}$$
where $[\widetilde{T}_HC_N]_{\mathbb Q}$ consists of the matrices $A\in\widetilde{T}_HC_N$ having rational entries.
\end{conjecture}

More precisely, this result seems to hold for the Fourier matrices, and we conjecture that this result holds in fact in general, under the ``regularity'' assumption in \cite{bbs}.

For usual Fourier matrices $F_N$, the result definitely holds at $N=p$ prime, because the minimal polynomial of $w$ over $\mathbb Q$ is simply $P=1+w+\ldots+w^{p-1}$. The case $N=p^2$ has also a simple solution, coming from the fact that all the $p\times p$ blocks of our matrix $A$ can be shown to coincide. In general, this method should probably work for $N=p^k$, or even for any $N\in\mathbb N$. However, since we don't have a complete proof here, the Fourier matrix statement should be regarded as being part of the above conjecture.

\section{Finite group calculations}

In this section we complete the computation of the defect of $F_G$, by using the formula in Theorem 4.2. It is convenient to consider the quantity $\delta(F_G)=|G|^{-1}d(F_G)$, which behaves better, and to try to compute it in the abstract group framework:

\begin{definition}
Associated to a finite group $G$ is the following quantity:
$$\delta(G)=\sum_{g\in G}\frac{1}{ord(g)}$$
\end{definition}

As a first example, consider a cyclic group $G=\mathbb Z_N$, with $N=p^a$ power of a prime. The count here is very simple, over sets of elements having a given order:
$$\delta(\mathbb Z_{p^a})=1+(p-1)p^{-1}+(p^2-p)p^{-2}+\ldots+(p^a-p^{a-1})p^{-1}=1+a-\frac{a}{p}$$

In order to extend this kind of count to the general abelian case, we use two ingredients. First is the following lemma, which splits the computation over isotypic components:

\begin{lemma}
For any finite groups $G,H$ we have:
$$\delta(G\times H)\geq\delta(G)\delta(H)$$
In addition, if $(|G|,|H|)=1$, we have equality. 
\end{lemma}

\begin{proof}
Indeed, we have the following estimate:
\begin{eqnarray*}
\delta(G\times H)
&=&\sum_{gh}\frac{1}{ord(g,h)}
=\sum_{gh}\frac{1}{[ord(g),ord(h)]}\\
&\geq&\sum_{gh}\frac{1}{ord(g)\cdot ord(h)}
=\delta(G)\delta(H)
\end{eqnarray*}

Now in the case $(|G|,|H|)=1$, the least common multiple appearing on the right becomes a product, $[ord(g),ord(h)]=ord(g)\cdot ord(h)$, so we have equality. 
\end{proof}

\begin{proposition}
For a finite abelian group $G$ we have
$$\delta(G)=\prod_p\delta(G_p)$$
where $G_p$ with $G=\times_pG_p$ are the isotypic components of $G$.
\end{proposition}

\begin{proof}
This is clear from Lemma 5.2, because the order of $G_p$ is a power of $p$.
\end{proof}

The second ingredient concerns the $p$-groups, and is as follows:

\begin{lemma}
For the $p$-groups, the quantities
$$c_k=\#\{g\in G|ord(g)\leq p^k\}$$
are multiplicative, in the sense that $c_k(G\times H)=c_k(G)c_k(H)$.
\end{lemma}

\begin{proof}
Indeed, for a product of $p$-groups we have:
\begin{eqnarray*}
c_k(G\times H)
&=&\#\{(g,h)|ord(g,h)\leq p^k\}\\
&=&\#\{(g,h)|ord(g)\leq p^k,ord(h)\leq p^k\}\\
&=&\#\{g|ord(g)\leq p^k\}\#\{h|ord(h)\leq p^k\}
\end{eqnarray*}

We recognize at right $c_k(G)c_k(H)$, and we are done.
\end{proof}

Let us compute now $\delta$ in the general isotypic case:

\begin{proposition}
For $G=\mathbb Z_{p^{a_1}}\times\ldots\times\mathbb Z_{p^{a_r}}$ with $a_1\leq a_2\leq\ldots\leq a_r$ we have
$$\delta(G)=1+\sum_{k=1}^rp^{(r-k)a_{k-1}+(a_1+\ldots+a_{k-1})-1}(p^{r-k+1}-1)[a_k-a_{k-1}]_{p^{r-k}}$$
with the convention $a_0=0$, and by using the notation $[a]_q=1+q+q^2+\ldots+q^{a-1}$.
\end{proposition}

\begin{proof}
First, in terms of the numbers $c_k$, we have:
$$\delta(G)=1+\sum_{k\geq 1}\frac{c_k-c_{k-1}}{p^k}$$

In the case of a cyclic group $G=\mathbb Z_{p^a}$ we have $c_k=p^{\min(k,a)}$. Thus, in the general isotypic case $G=\mathbb Z_{p^{a_1}}\times\ldots\times\mathbb Z_{p^{a_r}}$ we have:
$$c_k=p^{\min(k,a_1)}\ldots p^{\min(k,a_r)}=p^{\min(k,a_1)+\ldots+\min(k,a_r)}$$

Now observe that the exponent on the right is a piecewise linear function of $k$. More precisely, by assuming $a_1\leq a_2\leq\ldots\leq a_r$ as in the statement, the exponent is linear on each of the intervals $[0,a_1],[a_1,a_2],\ldots,[a_{r-1},a_r]$. So, the quantity $\delta(G)$ to be computed will be 1 plus the sum of $2r$ geometric progressions, 2 for each interval.

In practice now, the numbers $c_k$ are as follows:
$$c_0=1,c_1=p^r,c_2=p^{2r},\ldots,c_{a_1}=p^{ra_1},$$
$$c_{a_1+1}=p^{a_1+(r-1)(a_1+1)},c_{a_1+2}=p^{a_1+(r-1)(a_1+2)},\ldots,c_{a_2}=p^{a_1+(r-1)a_2},$$
$$c_{a_2+1}=p^{a_1+a_2+(r-2)(a_2+1)},c_{a_2+2}=p^{a_1+a_2+(r-2)(a_2+2)},\ldots,c_{a_3}=p^{a_1+a_2+(r-2)a_3},$$
$$\ldots\ldots\ldots$$
$$c_{a_{r-1}+1}=p^{a_1+\ldots+a_{r-1}+(a_{r-1}+1)},c_{a_{r-1}+2}=p^{a_1+\ldots+a_{r-1}+(a_{r-1}+2)},\ldots,c_{a_r}=p^{a_1+\ldots+a_r}$$

Now by separating the positive and negative terms in the above formula of $\delta(G)$, we have indeed $2r$ geometric progressions to be summed, as follows:
\begin{eqnarray*}
\delta(G)
&=&1+(p^{r-1}+p^{2r-2}+p^{3r-3}+\ldots+p^{a_1r-a_1})\\
&&-(p^{-1}+p^{r-2}+p^{2r-3}+\ldots+p^{(a_1-1)r-a_1})\\
&&+(p^{(r-1)(a_1+1)-1}+p^{(r-1)(a_1+2)-2}+\ldots+p^{a_1+(r-2)a_2})\\
&&-(p^{a_1r-a_1-1}+p^{(r-1)(a_1+1)-2}+\ldots+p^{a_1+(r-1)(a_2-1)-a_2})\\
&&+\ldots\\
&&+(p^{a_1+\ldots+a_{r-1}}+p^{a_1+\ldots+a_{r-1}}+\ldots+p^{a_1+\ldots+a_{r-1}})\\
&&-(p^{a_1+\ldots+a_{r-1}-1}+p^{a_1+\ldots+a_{r-1}-1}+\ldots+p^{a_1+\ldots+a_{r-1}-1})
\end{eqnarray*}

Now by performing all the sums, we obtain:
\begin{eqnarray*}
\delta(G)
&=&1+p^{-1}(p^r-1)\frac{p^{(r-1)a_1}-1}{p^{r-1}-1}\\
&&+p^{(r-2)a_1+(a_1-1)}(p^{r-1}-1)\frac{p^{(r-2)(a_2-a_1)}-1}{p^{r-2}-1}\\
&&+p^{(r-3)a_2+(a_1+a_2-1)}(p^{r-2}-1)\frac{p^{(r-3)(a_3-a_2)}-1}{p^{r-3}-1}\\
&&+\ldots\\
&&+p^{a_1+\ldots+a_{r-1}-1}(p-1)(a_r-a_{r-1})
\end{eqnarray*}

By looking now at the general term, we get the formula in the statement.
\end{proof}

Let us go back now to the formula in Theorem 4.2. By putting it together with the various results in this section, we obtain our main result in this paper, namely:

\begin{theorem}
For a finite abelian group $G$, decomposed as $G=\times_pG_p$, we have
$$d(F_G)=|G|\prod_p\left( 1+\sum_{k=1}^rp^{(r-k)a_{k-1}+(a_1+\ldots+a_{k-1})-1}(p^{r-k+1}-1)[a_k-a_{k-1}]_{p^{r-k}}\right)$$
where $a_0=0$ and $a_1\leq a_2\leq\ldots\leq a_r$ (depending on $p$) are such that $G_p=\mathbb Z_{p^{a_1}}\times\ldots\times\mathbb Z_{p^{a_r}}$.
\end{theorem}

\begin{proof}
Indeed, we know from Theorem 4.2 that we have $d(F_G)=|G|\delta(G)$, and the result follows from Proposition 5.3 and Proposition 5.5.
\end{proof}

As a first illustration, we can recover in this way the defect computation in \cite{tz2}:

\begin{corollary}
The undephased defect of a usual Fourier matrix $F_N$ is given by
$$d(F_N)=N\prod_{i=1}^s\left(1+a_i-\frac{a_i}{p_i}\right)$$
where $N=p_1^{a_1}\ldots p_s^{a_s}$ is the decomposition of $N$ into prime factors.
\end{corollary}

\begin{proof}
The underlying group here is the cyclic group $G=\mathbb Z_N$, whose isotypic components are the cyclic groups $G_{p_i}=\mathbb Z_{p_i^{a_i}}$. By applying now Theorem 5.6, we get:
$$d(F_N)=N\prod_{i=1}^s\left(1+p_i^{-1}(p_i-1)a_i\right)$$

But this is exactly the formula in the statement.
\end{proof}

As a second illustration, for the group $G=\mathbb Z_{p^{a_1}}\times\mathbb Z_{p^{a_2}}$ with $a_1\leq a_2$ we obtain:
\begin{eqnarray*}
d(F_G)
&=&p^{a_1+a_2}(1+p^{-1}(p^2-1)[a_1]_p+p^{a_1-1}(p-1)(a_2-a_1))\\
&=&p^{a_1+a_2-1}(p+(p^2-1)\frac{p^{a_1}-1}{p-1}+p^{a_1}(p-1)(a_2-a_1))\\
&=&p^{a_1+a_2-1}(p+(p+1)(p^{a_1}-1)+p^{a_1}(p-1)(a_2-a_1))
\end{eqnarray*}

In particular at $p=2$ and $a_1=a_2=1$ we obtain that the defect of the Fourier matrix $F_G=F_2\otimes F_2$, already known from Proposition 3.5 to be 10, is indeed:
$$d(F_G)=2(2+3+0)=10$$

Finally, let us mention that for general non-abelian groups, there doesn't seem to be any reasonable algebraic formula for the quantity $\delta(G)$. As an example, consider the dihedral group $D_N$, consisting of $N$ symmetries and $N$ rotations. We have:
$$\delta(D_N)=\frac{N}{2}+\delta(\mathbb Z_N)$$

Now by remembering the formula for $\mathbb Z_N$, namely $\delta(\mathbb Z_N)=\prod (1+p_i^{-1}(p_i-1)a_i)$, it is quite clear that the $N/2$ factor couldn't be incorporated in any nice way.

\section{Probabilistic speculations}

We have seen in the previous section that the defect of a Fourier matrix $F_G$ can be computed by doing some explicit calculations in the associated group $G$. In this section we speculate on a possible extension of this method, by using quantum groups.

The story here goes back to Popa's paper \cite{pop}, who discovered that the ``orthogonal MASA'' condition $\Delta\perp U\Delta U^*$ holds inside the $N\times N$ matrices precisely when $H=\sqrt{N}U$ is Hadamard. Such orthogonal MASA are known to produce subfactors, and Jones subsequently came with a complete study of the problem, from the point of view of statistical mechanics \cite{jo1}, general subfactor theory \cite{jsu}, and planar algebras \cite{jo2}.

Thanks to some general Tannakian correspondences, it was realized in the late 90's that these subfactors come in fact from certain quantum groups. More precisely, each complex Hadamard matrix $H\in M_N(\mathbb C)$ produces a certain quantum permutation group $G\subset S_N^+$, and the subfactor associated to $H$ can be understood in terms of $G$.

The construction $H\to G$ can be summarized as follows:

\begin{definition}
Let $H\in M_N(\mathbb C)$ be a complex Hadamard matrix.
\begin{enumerate}
\item Set $\xi_{ij}=H_i/H_j$, where $H_1,\ldots,H_N\in\mathbb T^N$ are the rows of $H$.

\item Let $P_{ij}\in M_N(\mathbb C)$ be the orthogonal rank one projection on $\xi_{ij}$.

\item Define a representation $\pi:C(S_N^+)\to M_N(\mathbb C)$ by $\pi(u_{ij})=P_{ij}$.

\item Consider the factorizations of type $\pi:C(S_N^+)\to C(G)\to M_N(\mathbb C)$.

\item Let $G\subset S_N^+$ be the minimal object producing such a factorization.
\end{enumerate}
\end{definition}

In this definition $S_N^+$ is the quantum permutation group of Wang \cite{wan}, known to be a compact quantum group in the sense of Woronowicz \cite{wor}. For technical details regarding this construction we refer to \cite{bbs}, \cite{bfs}, and for a recent survey on the subject, to \cite{ban}. 

The basic result here, which is for interest for us, is:

\begin{proposition}
The quantum permutation group associated to the Fourier matrix $F_G$ of a finite abelian group $G$ is nothing but the group $G$ itself, acting on itself.
\end{proposition}

\begin{proof}
The idea here is that examining the definition of $\pi$ leads to an obvious factorization of type $\pi:C(S_N^+)\to C(G)\to M_N(\mathbb C)$, with $N=|G|$. Now since $G$ is classical, the minimality property of this factorization is not hard to establish. See \cite{bbs}.
\end{proof}

This result, when combined with those in the previous section, raises the idea of computing the defect of an arbitrary complex Hadamard matrix $H\in M_N(\mathbb C)$ by using the associated quantum permutation group $G\subset S_N^+$. We do not know if this is possible:

\begin{question}
Does the quantum permutation group $G\subset S_N^+$ see the defect of the complex Hadamard matrix $H\in M_N(\mathbb C)$? And if so, how exactly?
\end{question}

In what follows we will just speculate on the second question. We recall from section 4 that in the case of Fourier matrices we have indeed a formula, namely:
$$d(H)=\sum_{g\in G}\frac{|G|}{ord(g)}$$

The problem is that the quantity on the right is not available for arbitrary quantum groups $G\subset S_N^+$, simply because these quantum groups are abstract objects, having no elements $g\in G$. So, let us pause now from our defect investigation, and remember what exact numeric quantities are available, for an arbitrary quantum group $G\subset S_N^+$.

The answer comes from Woronowicz's paper \cite{wor}, who developed there an analogue of the Peter-Weyl theory for the compact quantum Lie groups. In particular, Woronowicz established an analogue of the following key representation theory formula:
$$\int_GTr(g)^kdg=dim(Fix(u^{\otimes k}))$$

Now in the case of the usual symmetric group $S_N\subset O_N$, the character $Tr:G\to\mathbb R$ on the left is nothing but the number of fixed points. So, as a conclusion, we can say that ``a quantum permutation group $G\subset S_N^+$ doesn't exist as a concrete object, but the fixed point statistics $Tr:G\to\mathbb R$ does exist as a noncommutative random variable, and its moments can be computed by using representation theory methods''. 

As an example, a well-known result in classical probability tells us that in the limit $N\to\infty$, the character $Tr:S_N\to\mathbb N$ follows the Poisson law. By using representation theory methods, one can prove that in the limit $N\to\infty$ the character $Tr:S_N^+\to\mathbb R$ follows a free Poisson law, in the sense of free probability theory \cite{vdn}. See \cite{bco}.

In view of this observation, the following problem appears:

\begin{problem}
In the case of Fourier matrices $H=F_G$, can one recover the defect $d(H)$ from the fixed point statistics $Tr:G\to\mathbb N$ of the underlying group?
\end{problem}

The answer here is of course ``yes'', due for instance to the following formula:
$$ord(g)=\min\{r\in\mathbb N|Tr(g^r)=N\}$$

In order to deduce a formula for the defect, it is convenient to use the normalized trace $tr=Tr/N$. The first remark is that the ``min'' can be replaced by two parameters:
$$\frac{1}{ord(g)}=\lim_{k\to\infty}\lim_{l\to\infty}\frac{1}{l}\sum_{r=1}^ltr(g^r)^k$$

Indeed, with $k\to\infty$ only the maximal values of $tr$, namely $tr=1$, contribute to the computation of the limit, and these appear with a frequency of $1/ord(g)$. 

Now by summing over all the group elements, we obtain:
\begin{eqnarray*}
\sum_{g\in G}\frac{1}{ord(g)}
&=&\sum_{g\in G}\lim_{k\to\infty}\lim_{l\to\infty}\frac{1}{l}\sum_{r=1}^ltr(g^r)^k\\
&=&\lim_{k\to\infty}\lim_{l\to\infty}\frac{1}{l}\sum_{r=1}^l\sum_{g\in G}tr(g^r)^k\\
&=&N\lim_{k\to\infty}\lim_{l\to\infty}\frac{1}{l}\sum_{r=1}^l\int_Gtr(g^r)^k dg
\end{eqnarray*}

The point now is that the variables $\chi_r(g)=tr(g^r)$ are familiar objects, called Diaconis-Shahshahani type variables for $G$. See \cite{dsh}. Thus, we obtain the following result:

\begin{theorem}
For $G\subset S_N$ abelian acting on itself, with $N=|G|$, we have
$$d(F_G)=N^2\lim_{k\to\infty}\lim_{l\to\infty}\frac{1}{l}\sum_{r=1}^l\int_G\chi_r^k$$
where $\chi_r(g)=tr(g^r)$ are the Diaconis-Shahshahani type variables for $G$.
\end{theorem}

\begin{proof}
This is clear from the above considerations.
\end{proof}

The point now is that for an arbitrary quantum permutation group $G\subset S_N^+$ the variables $\chi_r$ are also available, cf. \cite{bcs}. So, we have the following question:

\begin{question}
Given a complex Hadamard matrix $H\in M_N(\mathbb C)$, is there a formula for $d(H)$ in terms of the variables $\chi_r$ over the associated quantum group $G\subset S_N^+$?
\end{question}

Finally, let us mention that the planar algebra associated to $H$ is known to be given by $P_k=Fix(u^{\otimes k})$, so that the planar algebra dimensions are:
$$\dim P_k=\int_G\chi_1^k$$

As a conclusion, it might happen that the planar algebra dimensions don't see the defect, but the Diaconis-Shahshahani type variables, taken together, do see it.

\end{document}